\documentclass[reqno, 12pt]{amsart}

\usepackage{mathrsfs}
\usepackage{amscd}
\usepackage{amsmath}
\usepackage{latexsym}
\usepackage{amsfonts}
\usepackage{amssymb}
\usepackage{amsthm}
\usepackage{graphicx}
\usepackage{tikz}
\usetikzlibrary{decorations.pathreplacing}
\usepackage{float}
\usepackage{hyperref}
\IfFileExists{makecell.sty}{\usepackage{makecell}}{} 
\usepackage{array}
\usepackage{booktabs}
\usepackage{multirow}
\usepackage{color,xcolor}

\newtheorem{theorem}{Theorem}[section]
\newtheorem{proposition}[theorem]{Proposition}
\newtheorem{corollary}[theorem]{Corollary}
\newtheorem{lemma}[theorem]{Lemma}

\newtheorem{remark}[theorem]{Remark}

\numberwithin{equation}{section}

\title[Inverse source problem]{Stability for the inverse source problem of the stochastic fractional Helmholtz equation}

 \author[D. Qiu]{Dong Qiu}
  \address{School of Mathematical Sciences, Zhejiang University, Hangzhou 310058, China}
\email{qiudong@zju.edu.cn}

  \author[X. Xu]{Xiang Xu}
  \address{School of Mathematical Sciences, and Center for Interdisciplinary Applied Mathematics, Zhejiang University, Hangzhou 310058, China}
\email{xxu@zju.edu.cn}
  
 \author[Y. Zhao]{Yue Zhao}
\address{School of Mathematics and Statistics, and Key Lab NAA-MOE, Central China Normal University,
Wuhan 430079, China}
\email{zhaoy@ccnu.edu.cn}

\subjclass[2020]{35R30, 78A46.}
\keywords{fractional Helmholtz equation, inverse source problem, white noise, stability}

\hypersetup{hidelinks}
\begin{document}

\begin{abstract}
This paper is concerned with the inverse source problem for the stochastic fractional Helmholtz equation driven by white noise.  For every fractional order $0<\alpha<1$, we prove the existence and uniqueness of the outgoing distributional solution to the direct problem with resolvent estimates at high frequencies, and establish its stochastic representation. For the inverse problem, we demonstrate that the variance of the source can be uniquely determined by the correlated random exterior data at a single frequency.
We further establish an increasing stability estimate for the inverse problem by using the multifrequency correlated exterior data. Our stability result shows that 
as the upper bound of the bandwidth of the utilized frequency increases, the stability will also improve.  
 The analysis employs the construction of geometric optics solutions, which connects the correlated data to the X-ray transform of the variance. 
\end{abstract}

\maketitle

\section{Introduction}

The fractional Helmholtz operator arises in the fixed-energy formulation of the fractional Schr\"odinger equation, developed by Laskin through a generalization of Feynman path integrals involving L\'evy paths \cite{Laskin1,Laskin2}. The fractional Laplacian also appears in models of anomalous diffusion and random motion with long-range jumps \cite{BV,Ros}. These connections have motivated extensive research on the analysis and numerical approximation of nonlocal partial differential equations involving the fractional Laplacian, see e.g. \cite{bl,sisc,S1,S2,TT, ZCV} and references cited therein. 

In uncertain environments, the source or the properties of the surrounding medium may fluctuate randomly, making stochastic models necessary to describe the resulting wave fields. The low regularity of random coefficients and sources creates additional difficulties in the analysis of the direct scattering problem. For the corresponding inverse problems, one seeks to determine statistical properties of the unknown source or medium, such as its mean and covariance, from measurements of the radiated field.

In this paper, we consider the fractional Helmholtz equation
\begin{align}\label{fhelm}
(-\Delta)^\alpha u(x)-k^{2\alpha}u(x)=f(x),\qquad x\in\mathbb R^3,
\end{align}
where $0<\alpha<1$, $k>0$ is the frequency, and $f=\sigma\dot W$ is the random source. Here $\dot W$ denotes real spatial white noise, which is formally the derivative of a three-dimensional Brownian sheet. The real-valued noise amplitude $\sigma$ is assumed to satisfy
\[
\sigma\in C_c^\infty(\mathbb R^3),\qquad
\operatorname{supp}\sigma\subset\subset B_r:=\{x\in\mathbb R^3:|x|<r\}.
\]
The variance density of the random source is $\sigma^2$.  It is known that
$f\in W^{-3/2-\epsilon,p}(\mathbb R^3)$ almost surely for every $p>1$ and $\epsilon>0$; see \cite{li2021inverse}.

The fractional Laplacian is defined through the Fourier transform by
\[
(-\Delta)^\alpha u
=\mathcal F^{-1}\{|\xi|^{2\alpha}\widehat u(\xi)\},
\qquad u\in H^{2\alpha}(\mathbb R^3).
\]
For sufficiently regular, decaying functions, it also admits the singular integral representation
\[
(-\Delta)^\alpha u(x)
=C(\alpha)\operatorname{p.v.}\int_{\mathbb R^3}
 \frac{u(x)-u(y)}{|x-y|^{3+2\alpha}}\,dy,
\]
where $C(\alpha)>0$ is a normalizing constant \cite{HSZ}. The distributional interpretation of this operator is given  in Section \ref{sec:green}. We consider the outgoing solution by imposing the Sommerfeld radiation condition
\begin{align}\label{src}
\lim_{r\to\infty}r(\partial_ru-iku)=0,\qquad r=|x|,
\end{align}
uniformly in the observation direction $\hat{x}=x/|x|$.
In this paper, we study both the direct and inverse problems for the fractional Helmholtz equation. Given the random source, the direct problem is to determine the outgoing random field and establish its existence, uniqueness, and energy estimates in suitable rough function spaces. The inverse problem is to recover the variance $\sigma^2$ from correlated exterior data.

Inverse source problems arise in many scientific and industrial applications, including antenna design and synthesis, biomedical imaging, and photoacoustic tomography \cite{arridge,blz,Isakov_1}. Inverse random source and potential problems for acoustic, elastic, and electromagnetic waves have been studied extensively \cite{LPS,LLM,LLW,li2021inverse,WXZ_1,WXZ}. For stochastic inverse scattering problems, correlations of the radiated random field usually carry information about the statistical properties of the source or medium, and multifrequency measurements provide a way to overcome the ill-posedness and improve stability of the inverse source problems.

There has also been progress on inverse problems involving fractional operators. Fractional analogues of the Calder\'on problem were investigated in \cite{grsu,gsu,ruland}, and simultaneous recovery of a source and a potential was considered in \cite{CL}. For the fractional Helmholtz equation, the recent work \cite{LL} studied direct and inverse scattering with a microlocally isotropic generalized Gaussian random source, whose covariance is a classical pseudodifferential operator. A deterministic inverse potential scattering problem was considered in \cite{msq}. Both of these works use far-field data at all high frequencies. However, to the best of our knowledge,
no results are available for the case of using the exterior data for scattering problems of the fractional Helmholtz equations.

For the direct problem, the principal difficulty is the distributional character of the white-noise source. We first construct the outgoing solution as the convolution of the fractional Green function with $\sigma\dot W$. The local behavior of the kernel gives the Sobolev regularity of the solution, while a reduction to the classical Helmholtz equation yields uniqueness by the Rellich theorem. This establishes a unique outgoing distributional solution for every $0<\alpha<1$. 
We also continue the truncated fractional resolvent holomorphically from the first quadrant to the right half-plane and prove high-frequency estimates on rough Sobolev spaces. The proof relies on the analysis of the Fourier multiplier.

When $3/4<\alpha<1$, we show that the stochastic representation also defines a continuous mild solution on the whole space. Following the constructive approach of \cite[Theorem 2.7]{BCL2016}, we use estimates for the Green function and its increments to establish continuity, and then obtain the mild solution as the limit of solutions driven by piecewise constant approximations of the noise. The restriction $\alpha>3/4$ is dictated by the local square integrability of the Green function at points where $\sigma$ is nonzero. It should be noticed that the restriction of the distributional solution to the exterior domain away from the support of the source is smooth. Thus, the exterior correlation data for the inverse problem are   well defined for all $0<\alpha<1$. 

For the inverse problem, by using correlated measurements $\mathbb E[u(x,k)u(z_0,k)]$ with a fixed measurement point $z_0$ and $x\in U$ with $U$ being a bounded exterior domain,
we prove the uniqueness at a single frequency. The key ingredient in the proof is an application of the unique continuation property of the fractional Laplacian given by \cite[Theorem 1.2]{gsu}. We next establish a quantitative stability estimate by using the multifrequency correlated data on an exterior shell and a surrounding sphere. It is worthing mentioning  that the traditional boundary integration-by-parts identities used for classical wave equations in \cite{blz,LZZ} do not directly reduce the current nonlocal problem to surface measurements. Thus, different approaches shall be developed for the fractional equations. To deal with this issue, we combine the high-frequency expansion of the fractional Green function with the geometric optics construction to relate the correlated data to the X-ray transform of the variance.  Moreover, we employ an analytic continuation principle to derive a logarithmic-type stability estimate for every fixed $0<\alpha<1$. The result shows that as the upper bound of the bandwidth of the given frequency increases, the logarithmic stability diminishes and approaches a Lipschitz one.

The paper is organized as follows. Section \ref{dp} treats the direct problem. We first derive estimates for the Green function and establish the well-posedness of the direct problem. Then 
we give a constructive proof for the continuous mild solution. Section \ref{ip} is devoted to the single-frequency uniqueness and multifrequency increasing stability for the inverse random source problem. The proofs of the resolvent estimates are collected in the Appendix.

\section{Direct problem}\label{dp}

In this section, we establish the well-posedness of the direct problem and obtain the stochastic representation used in the inverse analysis. We begin with the outgoing Green function and the distributional solution, then construct the continuous mild solution, and conclude with energy estimates that describe the dependence on the wavenumber.

\subsection{Green function and outgoing distributional solution}\label{sec:green}

 Fix $k>0$ and a bounded cube $D$ containing $\operatorname{supp}\sigma$ in its interior. We regard $W$ as real isonormal white noise on $L^2(D;\mathbb R)$ and extend it complex linearly to complex-valued integrands. In particular,
\begin{equation}\label{eq:mild-isometry}
\mathbb E\left|\int_D h(y)\,dW_y\right|^2=\|h\|_{L^2(D)}^2,
\qquad h\in L^2(D;\mathbb C).
\end{equation}
The source $f=\sigma\dot W$ is defined by
$\langle f,\varphi\rangle=\int_D\sigma(y)\varphi(y)\,dW_y$.
We use the Fourier transform
$\widehat\varphi(\xi)=\int_{\mathbb R^3}e^{-ix\cdot\xi}\varphi(x)\,dx$.
Set
\begin{equation}\label{eq:mild-green}
 p_{\alpha,k}(\xi)=|\xi|^{2\alpha}-k^{2\alpha},\qquad
 G_{\alpha,k}=\mathcal F^{-1}m^+_{\alpha,k},\qquad
 m^+_{\alpha,k}=(p_{\alpha,k}-i0)^{-1}.
\end{equation}
Here
\[
(p_{\alpha,k}-i0)^{-1}
=\lim_{\varepsilon\downarrow0}
\frac{1}{|\xi|^{2\alpha}-k^{2\alpha}-i\varepsilon}
\quad\text{in }\mathcal S'(\mathbb R^3).
\]
This limiting absorption prescription selects the outgoing solution, and
$p_{\alpha,k}m^+_{\alpha,k}=1$ gives
$\bigl((-\Delta)^\alpha-k^{2\alpha}\bigr)G_{\alpha,k}=\delta_0$.
The products near $\xi=0$ are understood as products of locally integrable functions, and near the spectral sphere the symbol is smooth. We write $G_{\alpha,k}(x-y)$ for the corresponding kernel; it is the function denoted by $G_k(x,y)$ in Section~\ref{ip}.

The following lemma gives the local and far-field estimates for the outgoing Green function. The constants in this subsection and in Section \ref{sec:mild} may depend on the fixed $\alpha$ and $k$.

\begin{lemma}\label{lem:mild-kernel}
Let $0<\alpha<1$. The kernel $G_{\alpha,k}$ is locally integrable and smooth away from the origin, and
\begin{equation}\label{eq:mild-singularity}
G_{\alpha,k}(z)=c_\alpha|z|^{2\alpha-3}(1+o(1)),
\qquad |z|\rightarrow0,
\end{equation}
where $c_\alpha=\sin(\pi\alpha)\Gamma(2-2\alpha)/(2\pi^2)>0$.
If $3/4<\alpha<1$, then, for every bounded cube $Q$ and every
$0<\beta<2\alpha-3/2$,
\begin{align}
 \sup_{x\in Q}\int_D|G_{\alpha,k}(x-y)|^2\,dy&\le C,
 \label{eq:mild-kernel-L2}\\
 \int_D|G_{\alpha,k}(x-y)-G_{\alpha,k}(z-y)|^2\,dy
 &\le C_\beta|x-z|^{2\beta},\qquad x,z\in Q.
 \label{eq:mild-kernel-increment}
\end{align}
Moreover, if $y$ ranges over a fixed compact set, the kernel and each of its $y$-derivatives satisfy
\begin{align}
 \partial_y^\nu G_{\alpha,k}(r\theta-y)&=O(r^{-1}),\label{eq:mild-kernel-decay}\\
 (\partial_r-ik)\partial_y^\nu G_{\alpha,k}(r\theta-y)&=O(r^{-2}),\label{eq:mild-kernel-radiation}
\end{align}
uniformly in $\theta\in\mathbb S^2$ as $r\to\infty$.
\end{lemma}

\begin{proof}
The scalar fractional resolvent formula, with the outgoing boundary value at the positive spectral parameter, gives
\begin{equation}\label{eq:mild-green-decomposition}
G_{\alpha,k}(z)=\frac{k^{2-2\alpha}}{\alpha}
\frac{e^{ik|z|}}{4\pi|z|}+J_{\alpha,k}(|z|),
\end{equation}
where
\begin{equation}\label{eq:mild-J}
J_{\alpha,k}(r)=\frac{\sin(\pi\alpha)}{4\pi^2r}
\int_0^\infty
\frac{t^\alpha e^{-r\sqrt t}}
{t^{2\alpha}-2k^{2\alpha}t^\alpha\cos(\pi\alpha)+k^{4\alpha}}\,dt.
\end{equation}
This is the kernel form of the fractional power identity in
\cite[(5.28)]{MS}; see also the outgoing Green function analysis in \cite{LL}.
The denominator is bounded below by
$c_\alpha'(t^{2\alpha}+k^{4\alpha})$ for some $c_\alpha'>0$.
Substituting $s=r\sqrt t$ in \eqref{eq:mild-J} and applying dominated convergence yields
\[
\lim_{r\downarrow0}r^{3-2\alpha}J_{\alpha,k}(r)
=\frac{\sin(\pi\alpha)}{2\pi^2}
\int_0^\infty s^{1-2\alpha}e^{-s}\,ds=c_\alpha.
\]
The classical term in \eqref{eq:mild-green-decomposition} is of lower singular order, proving \eqref{eq:mild-singularity}. Local integrability follows since $2\alpha>0$. Differentiation under the integral in \eqref{eq:mild-J} is valid away from $r=0$. The same substitution, now for $r\ge1$, gives
$|\partial_z^\nu J_{\alpha,k}(|z|)|\le C_\nu|z|^{-3-2\alpha-|\nu|}$.
Together with the classical outgoing kernel this proves
\eqref{eq:mild-kernel-decay}--\eqref{eq:mild-kernel-radiation}.

For the increment estimate, choose $\chi\in C_c^\infty(\mathbb R^3)$ equal to one on a neighborhood of $\{|\xi|\le 2k+1\}$ and write
\[
G_{\alpha,k}=S+V,\qquad
S=\mathcal F^{-1}(\chi m^+_{\alpha,k}),\qquad
V=\mathcal F^{-1}b,\qquad
b=\frac{1-\chi}{|\xi|^{2\alpha}-k^{2\alpha}}.
\]
The inverse Fourier transform $S$ of a compactly supported distribution is smooth. On bounded sets its contribution to the left-hand side of \eqref{eq:mild-kernel-increment} is bounded by $C|x-z|^2$. The high-frequency multiplier satisfies
$|b(\xi)|\le C\langle\xi\rangle^{-2\alpha}$, where $\langle\xi\rangle = (1+|\xi|^2)^{1/2}$.
Consequently $b\in L^2$ when $4\alpha>3$, and Plancherel's identity gives
\begin{align*}
\|V(x-\cdot)-V(z-\cdot)\|_{L^2(\mathbb R^3)}^2
&\le C|x-z|^{2\beta}
\int_{|\xi|>1}|\xi|^{2\beta-4\alpha}\,d\xi\\
&\le C_\beta|x-z|^{2\beta},
\end{align*}
provided $\beta<2\alpha-3/2$.
This proves \eqref{eq:mild-kernel-increment}. The same decomposition also gives \eqref{eq:mild-kernel-L2}.
\end{proof}

We now apply these estimates to the random source and construct the corresponding outgoing distributional solution. For a tempered distribution whose Fourier transform is locally integrable near the origin, multiplication by $p_{\alpha,k}$ is understood as ordinary function multiplication near zero and distributional multiplication away from zero. A smooth cutoff combines these two definitions, which agree on their overlap. This specifies the distributional interpretation of \eqref{fhelm} in the solution class below.

\begin{proposition}\label{prop:mild-exterior}
Let $0<\alpha<1$, $k>0$, and $\sigma\in C_c^\infty(\mathbb R^3)$. There is a unique outgoing random distribution
\begin{equation}\label{eq:mild-distribution}
u=G_{\alpha,k}*(\sigma\dot W),
\end{equation}
and, almost surely,
\begin{equation}\label{eq:mild-Sobolev}
u\in H^t_{\rm loc}(\mathbb R^3),\qquad t<2\alpha-\frac32.
\end{equation}
Its restriction to $\mathbb R^3\setminus\operatorname{supp}\sigma$ has a smooth modification given by
\begin{equation}\label{eq:mild-exterior-integral}
u(x,k)=\int_DG_{\alpha,k}(x-y)\sigma(y)\,dW_y,
\qquad x\notin\operatorname{supp}\sigma.
\end{equation}
On one event of probability one, this distribution solves \eqref{fhelm}, is smooth outside a ball, and satisfies
\begin{equation}\label{eq:mild-outgoing-class}
\sup_{\theta\in\mathbb S^2}|u(r\theta)|=O(r^{-1}),\qquad
\sup_{\theta\in\mathbb S^2}r|\partial_ru(r\theta)-iku(r\theta)|\rightarrow0.
\end{equation}
It is unique among tempered distributional solutions whose Fourier transforms are locally integrable near the origin, which are smooth outside a ball and satisfy \eqref{eq:mild-outgoing-class}.
\end{proposition}

\begin{proof}
The noise admits a compactly supported $H^{-s}(\mathbb R^3)$-valued realization for every $s>3/2$. Indeed, an orthonormal basis construction and the isometry \eqref{eq:mild-isometry} give
\[
\mathbb E\|\sigma\dot W\|_{H^{-s}(\mathbb R^3)}^2
=(2\pi)^{-3}\|\sigma\|_{L^2(D)}^2
\int_{\mathbb R^3}\langle\xi\rangle^{-2s}\,d\xi<\infty.
\]
The realizations can be chosen to agree as distributions, with support contained in $\operatorname{supp}\sigma$. Denote this realization by $F$. Since $F$ has compact support, \eqref{eq:mild-distribution} defines a tempered distribution, with
\[
\widehat u=m^+_{\alpha,k}\widehat F.
\]
In the decomposition used in Lemma~\ref{lem:mild-kernel}, $S*F$ is smooth and the high-frequency multiplier $b$ maps $H^{-s}$ into $H^{2\alpha-s}$. Taking $s>3/2$ arbitrarily close to $3/2$ proves \eqref{eq:mild-Sobolev}.

The last Fourier identity also verifies the equation without multiplying a general distribution by a symbol nonsmooth at zero. Near zero, $m^+_{\alpha,k}\widehat F$ is an ordinary locally bounded function and $p_{\alpha,k}$ is nonzero. Away from zero, $p_{\alpha,k}$ is smooth and
$p_{\alpha,k}m^+_{\alpha,k}=1$ in the distributional sense. Using a cutoff near zero therefore gives
$p_{\alpha,k}\widehat u=\widehat F$, which is the Fourier interpretation of \eqref{fhelm} for this solution.

Let $K$ be a compact subset of the complement of $\operatorname{supp}\sigma$. The kernels
$\sigma(y)\partial_x^\nu G_{\alpha,k}(x-y)$ are square integrable in $y$, uniformly for $x\in K$. The stochastic integral in \eqref{eq:mild-exterior-integral} and all its derivatives are thus defined in mean square. Applying the isometry on bounded open neighborhoods of $K$ gives finite expected Sobolev norms of every integer order. Sobolev embedding and a countable exhaustion yield a smooth modification. Stochastic Fubini, tested against compactly supported smooth functions on these neighborhoods, identifies this modification with $G_{\alpha,k}*F$. The identities hold simultaneously for all tests by taking a countable dense set first. Finally, pairing the kernel with the compactly supported, finite-order distribution $F$ and using \eqref{eq:mild-kernel-decay}--\eqref{eq:mild-kernel-radiation} proves \eqref{eq:mild-outgoing-class} pathwise.

For uniqueness, let $w$ be the difference of two solutions in the stated class. The equation $p_{\alpha,k}\widehat w=0$ first implies that $\widehat w=0$ near zero, where it is a locally integrable function. Away from zero its support is contained in $|\xi|=k$. On a neighborhood of this sphere the quotient
\[
\frac{|\xi|^2-k^2}{|\xi|^{2\alpha}-k^{2\alpha}}
\]
extends to a smooth function. Multiplying the equation by this quotient, with a cutoff around the sphere, gives
$(|\xi|^2-k^2)\widehat w=0$ on the whole space. Hence $w$ is a smooth entire solution of the classical homogeneous Helmholtz equation. The Sommerfeld condition in \eqref{eq:mild-outgoing-class} and the classical Rellich uniqueness theorem give $w=0$.
\end{proof}

\subsection{Continuous mild solution}\label{sec:mild}

 The exterior representation in Proposition \ref{prop:mild-exterior} provides the pointwise field needed for the inverse problem. We now show that, for $3/4<\alpha<1$, this representation defines a continuous random field throughout $\mathbb R^3$. We refer to this stochastic Green representation as the mild solution of \eqref{fhelm}. The proof follows the constructive approach of \cite[Theorem 2.7]{BCL2016}, using the fractional kernel estimates established above.

\begin{theorem}\label{thm:mild-continuous}
Let $3/4<\alpha<1$, $k>0$, and let $\sigma\in C_c^\infty(\mathbb R^3)$ be real-valued. There exists a continuous random field given by
\begin{equation}\label{eq:mild-integral}
u(x,k)=\int_DG_{\alpha,k}(x-y)\sigma(y)\,dW_y,
\qquad x\in\mathbb R^3.
\end{equation}
For the same white noise, this field is unique up to indistinguishability among continuous fields satisfying \eqref{eq:mild-integral}. It represents the distributional solution in Proposition \ref{prop:mild-exterior} and satisfies its outgoing conditions. Moreover, for every bounded cube $Q$, every $0<\gamma<2\alpha-3/2$, and every $1\le p<\infty$,
\begin{equation}\label{eq:mild-holder}
\mathbb E\|u(\cdot,k)\|_{C^{0,\gamma}(Q)}^p<\infty.
\end{equation}
The mild solution is also the mean-square $L^2(Q)$ limit of the outgoing solutions corresponding to piecewise constant approximations of the white noise.
\end{theorem}

\begin{proof}
We first establish continuity of the stochastic integral in \eqref{eq:mild-integral}. By \eqref{eq:mild-kernel-L2}, its integrand belongs to $L^2(D)$ for every $x$, so $u(x,k)$ is a well-defined centered complex Gaussian random variable. Fix a bounded cube $Q$ and suppress $k$ in the notation. For $x,z\in Q$ and $0<\beta<2\alpha-3/2$, the white-noise isometry and \eqref{eq:mild-kernel-increment} give
\begin{equation}\label{eq:mild-stochastic-increment}
\begin{split}
\mathbb E|u(x)-u(z)|^2
&=\int_D|G_{\alpha,k}(x-y)-G_{\alpha,k}(z-y)|^2\sigma^2(y)\,dy\\
&\le C_\beta|x-z|^{2\beta}.
\end{split}
\end{equation}
Since the increment is Gaussian, for every $q\ge2$,
\[
\mathbb E|u(x)-u(z)|^q\le C_{q,\beta}|x-z|^{q\beta}.
\]
Given $0<\gamma<2\alpha-3/2$, choose $\gamma<\beta<2\alpha-3/2$ and then $q$ so large that $q(\beta-\gamma)>3$. Kolmogorov's continuity theorem gives a $\gamma$-H\"older modification on $Q$. A countable exhaustion by cubes, together with a countable sequence of exponents approaching $2\alpha-3/2$, gives a common continuous modification with every asserted local H\"older exponent.

The same estimate gives the moment bound \eqref{eq:mild-holder}. For a prescribed $p$, choose $q\ge\max\{2,p\}$ sufficiently large and $s$ such that $\gamma+3/q<s<\beta$. Tonelli's theorem yields
\[
\mathbb E\|u\|_{W^{s,q}(Q)}^q
\le C+C\int_Q\int_Q|x-z|^{q(\beta-s)-3}\,dx\,dz<\infty.
\]
The embedding $W^{s,q}(Q)\hookrightarrow C^{0,\gamma}(Q)$ and H\"older's inequality prove the claim. Uniqueness follows from the representation: two continuous fields satisfying \eqref{eq:mild-integral} for the same $W$ agree almost surely at each point of a countable dense set, and hence agree everywhere on a single event of probability one.

We next construct the mild solution by approximating the noise. Let $\{D_{n,j}\}_{j=1}^{N_n}$ be any sequence of finite measurable partitions of $D$ into sets of positive measure, with
\[
h_n:=\max_{1\le j\le N_n}\operatorname{diam}(D_{n,j})\rightarrow0.
\]
Define the piecewise constant random functions
\begin{equation}\label{eq:mild-noise-approximation}
\dot W_n(y)=\sum_{j=1}^{N_n}
\frac{W(D_{n,j})}{|D_{n,j}|}\mathbf 1_{D_{n,j}}(y),
\qquad W(D_{n,j})=\int_{D_{n,j}}dW_y.
\end{equation}
The isometry gives $\mathbb E|W(D_{n,j})|^2=|D_{n,j}|$, and consequently
\[
\mathbb E\|\sigma\dot W_n\|_{L^2(D)}^2
=\sum_{j=1}^{N_n}\frac{1}{|D_{n,j}|}
\int_{D_{n,j}}\sigma^2(y)\,dy<\infty.
\]
Thus each regularized source belongs to $L^2(D)$ almost surely and has compact support. Its outgoing Green potential is
\[
u_n(x)=\int_DG_{\alpha,k}(x-y)\sigma(y)\dot W_n(y)\,dy,
\]
which solves
\[
\bigl((-\Delta)^\alpha-k^{2\alpha}\bigr)u_n=\sigma\dot W_n
\quad\text{in }\mathbb R^3
\]
and satisfies the outgoing conditions in \eqref{eq:mild-outgoing-class}.

To prove convergence, let $P_n$ be the orthogonal projection in $L^2(D)$ onto the functions constant on each $D_{n,j}$. The condition $h_n\to0$ and uniform continuity imply $P_nh\to h$ for $h\in C(\overline D)$. Density and the bound $\|P_n\|\le1$ extend this convergence to every $h\in L^2(D)$. Writing
\[
h_x(y)=G_{\alpha,k}(x-y)\sigma(y),
\]
we have $u_n(x)=W(P_nh_x)$ and $u(x)=W(h_x)$. Hence
\[
\mathbb E|u_n(x)-u(x)|^2
=\|(P_n-I)h_x\|_{L^2(D)}^2\rightarrow0.
\]
The last quantity is bounded by $\|h_x\|_{L^2(D)}^2$, whose integral over $Q$ is finite by \eqref{eq:mild-kernel-L2}. Dominated convergence therefore gives
\begin{equation}\label{eq:mild-approximation-limit}
\mathbb E\|u_n-u\|_{L^2(Q)}^2\rightarrow0.
\end{equation}

It remains to identify this continuous field with the distributional solution. For $\varphi\in C_c^\infty(\mathbb R^3)$, the kernel bounds and stochastic Fubini give
\[
\int_{\mathbb R^3}u(x)\varphi(x)\,dx
=\int_D\sigma(y)
\left(\int_{\mathbb R^3}G_{\alpha,k}(x-y)\varphi(x)\,dx\right)dW_y.
\]
The right-hand side is the pairing of $G_{\alpha,k}*(\sigma\dot W)$ with $\varphi$. Taking a countable dense family of test functions first identifies the two distributions on one event of probability one. Proposition \ref{prop:mild-exterior} then gives the equation and the outgoing conditions for the continuous mild solution.
\end{proof}

\begin{remark}\label{rem:mild-threshold}
The restriction $\alpha>3/4$ is necessary for the pointwise Wiener integral at any point $x_0$ where $\sigma(x_0)\ne0$. Indeed, the local singularity \eqref{eq:mild-singularity} gives, for sufficiently small $\varepsilon>0$,
\[
\int_{|y-x_0|<\varepsilon}
|G_{\alpha,k}(x_0-y)\sigma(y)|^2\,dy
\ge C\int_0^\varepsilon r^{4\alpha-4}\,dr.
\]
The integral diverges for $\alpha\le3/4$, with logarithmic divergence at equality. At exterior points the kernel is smooth on the support of $\sigma$, so Proposition \ref{prop:mild-exterior} still provides the stochastic representation for every $0<\alpha<1$.
\end{remark}

The stochastic representation now gives the following correlation identities used in the inverse problem.

\begin{corollary}\label{cor:mild-correlation}
Let $0<\alpha<1$ and $x,z\notin\operatorname{supp}\sigma$. The outgoing field satisfies
\begin{align}
\mathbb E[u(x,k)\overline{u(z,k)}]
&=\int_DG_{\alpha,k}(x-y)\overline{G_{\alpha,k}(z-y)}\sigma^2(y)\,dy,
\label{eq:mild-covariance}\\
\mathbb E[u(x,k)u(z,k)]
&=\int_DG_{\alpha,k}(x-y)G_{\alpha,k}(z-y)\sigma^2(y)\,dy.
\label{eq:mild-relation}
\end{align}
If $\alpha>3/4$, both identities hold for all $x,z\in\mathbb R^3$.
\end{corollary}
\begin{proof}
Apply the white-noise isometry to the stochastic representation. Since $W$ is real and its action on complex integrands is defined by complex linearity, it gives both the covariance \eqref{eq:mild-covariance} and the nonconjugated second moment \eqref{eq:mild-relation}.
\end{proof}

\subsection{High-frequency resolvent estimates}\label{sec:resolvent}

We retain the high-frequency estimates separately from the fixed-frequency construction. The proofs of Lemma \ref{lem1} and Proposition \ref{free-frac} are included in the Appendix.

In three spatial dimensions, the strong Huygens principle gives the following estimates for the classical free resolvent (cf.~\cite{DZ,WXZ}). Throughout the paper, $a\lesssim b$ means $a\leq Cb$, where the constant is independent of the frequency and the data discrepancy.

\begin{lemma}\label{lem1}
For every $s\in\mathbb R$, the free resolvent
\[
R_0(\lambda)=(-\Delta-\lambda^2)^{-1}:H^s(\mathbb R^3)\to H^{s+m}(\mathbb R^3),
\qquad 0\leq m\leq2,
\]
is analytic for $\Im\lambda>0$. For $\chi\in C_c^\infty(\mathbb R^3)$, the truncated resolvent extends to an entire operator family and satisfies
\begin{equation}\label{res1}
\|\chi R_0(\lambda)\chi\|_{H^s\to H^{s+m}}
\lesssim(1+|\lambda|)^{m-1}e^{L(\Im\lambda)_-},
\end{equation}
where $(\Im\lambda)_-=\max\{-\Im\lambda,0\}$ and $L>\operatorname{diam}(\operatorname{supp}\chi)$.
\end{lemma}

We now consider the fractional resolvent. Complex powers of $\lambda$ are taken on the right half-plane, with $-\pi/2<\arg\lambda<\pi/2$. For $0<\theta_0<\pi/2$, set
\[
S_{\theta_0}=\{\lambda\neq0:|\arg\lambda|\leq\theta_0\}.
\]
The outgoing fractional resolvent is related to the classical one by the fractional power formula \cite[(5.28)]{MS}:
\begin{equation}\label{formula2}
\begin{split}
R(\lambda)
={}&\frac{\lambda^{2-2\alpha}}{\alpha}R_0(\lambda)+\frac{\sin(\pi\alpha)}{\pi}\int_0^\infty
\frac{\gamma^\alpha(\gamma-\Delta)^{-1}}
{\gamma^{2\alpha}-2\gamma^\alpha\lambda^{2\alpha}\cos(\pi\alpha)+\lambda^{4\alpha}}\,d\gamma.
\end{split}
\end{equation}
 The integral is understood through its scalar Fourier multiplier, estimated in the Appendix. Initially this is the spectral resolvent for $\Re\lambda,\Im\lambda>0$. After spatial truncation, it defines its outgoing continuation to $\Re\lambda>0$.

\begin{proposition}\label{free-frac}
Let $0<\alpha<1$, $s\in\mathbb R$, and $\rho\in C_c^\infty(\mathbb R^3)$.
For every $0\leq l\leq2\alpha$, the outgoing truncated resolvent
\[
\rho R(\lambda)\rho:H^s(\mathbb R^3)\rightarrow H^{s+l}(\mathbb R^3)
\]
extends holomorphically to $\Re\lambda>0$. For every $\theta_0<\pi/2$ and
$T>\operatorname{diam}(\operatorname{supp}\rho)$, it satisfies
\begin{equation}\label{bound_2}
\|\rho R(\lambda)\rho\|_{H^s\to H^{s+l}}
\leq C|\lambda|^{1+l-2\alpha}e^{T(\Im\lambda)_-},
\qquad \lambda\in S_{\theta_0},\quad |\lambda|\geq1.
\end{equation}
The constant can depend on $\alpha,s,l,\rho,\theta_0,T$.
\end{proposition}

For real $k>0$, the kernel in \eqref{formula2} agrees with \eqref{eq:mild-green-decomposition}. Thus its action on the random source gives the same outgoing solution as Proposition \ref{prop:mild-exterior}. Combining that proposition with \eqref{bound_2} gives the following energy estimate.

\begin{theorem}
Let $0<\alpha<1$ and $\eta>0$. For every $k>0$, the direct problem
has a unique solution in the outgoing class of
Proposition~\ref{prop:mild-exterior}. Almost surely, it belongs to
$H^t_{\rm loc}(\mathbb R^3)$ for every $t<2\alpha-3/2$.
Moreover, for each fixed $\chi\in C_c^\infty(\mathbb R^3)$ and $k\geq1$,
\[
\|\chi u\|_{H^{-3/2-\eta+l}}
\leq C k^{1-2\alpha+l}\|f\|_{H^{-3/2-\eta}},
\qquad 0\leq l\leq2\alpha.
\]
\end{theorem}
\begin{proof}
Choose $\rho=1$ on a neighborhood of
$\operatorname{supp}\chi\cup\operatorname{supp}\sigma$ and apply
Proposition~\ref{free-frac}. Existence and uniqueness follow from
Proposition~\ref{prop:mild-exterior}. The choice $l=2\alpha$, followed
by a countable choice of $\eta\downarrow0$, gives the stated local
regularity on one probability-one event.
\end{proof}

\begin{remark}
The Sobolev gain in Proposition \ref{free-frac} is available throughout
$0\le l\le2\alpha$. Its high-frequency factor decays only when
$1-2\alpha+l<0$. In particular, the endpoint estimate has a factor $k$;
it does not improve the threshold for a pointwise continuous mild solution.
\end{remark}

\section{Inverse random source problem}\label{ip}

In this section, we first prove single-frequency uniqueness by fractional unique continuation and then obtain multifrequency stability using the high-frequency expansion and the X-ray transform.  Throughout the inverse analysis, $0<\alpha<1$ is fixed. Fix an integer $s\geq4$ and $F>0$, and let
\[
\mathcal C_F=\{v\in W^{s,\infty}(\mathbb R^3):
\|v\|_{W^{s,\infty}}\leq F,\quad\operatorname{supp}v\subset\subset B_r\}.
\]
 We assume that the variance $\sigma^2$ belongs to $\mathcal C_F$. For two sources, write
\[
\rho_j=\sigma_j^2,\qquad h=\rho_1-\rho_2\in\mathcal C_{2F}.
\]
The amplitudes $\sigma_j$ retain the smoothness assumed in the Introduction. 

Denote the classical Green function by
\[
G_k^{(0)}(x,z)=\frac{e^{ik|x-z|}}{4\pi|x-z|},
\]
and define the beam transform
\[
Bg(x,\theta)=\int_{-\infty}^0g(x+t\theta)\,dt,\qquad\theta\in\mathbb S^2.
\]
The next lemma gives the geometric optics expansion used below.

\begin{lemma}\label{he}
Let $g\in\mathcal C_F$, $r<r_1<R$, and $z\in\partial B_R$. For $k\geq1$ and $x\in B_{r_1}$,
\[
R_0(k)\big(gG_k^{(0)}(\cdot,z)\big)(x)
=\frac{i}{2k}G_k^{(0)}(x,z)Bg(x,\theta)+P_k(x,z),
\qquad \theta=\frac{x-z}{|x-z|},
\]
where
\[
\sup_{z\in\partial B_R}\|P_k(\cdot,z)\|_{L^2(B_{r_1})}\leq Ck^{-2}.
\]
The constant depends only on $F,r,r_1,R$.
\end{lemma}

 \begin{proof}
We adapt the transport construction in \cite{Stefanov}. By smooth approximation with uniformly bounded $W^{4,\infty}$ norms and support in $B_r$, it suffices to consider smooth $g$. Put $d_z(x)=|x-z|$, $\theta_z(x)=(x-z)/d_z(x)$, and
\[
A_z(x)=\int_0^{d_z(x)}g(z+t\theta_z(x))\,dt.
\]
This amplitude vanishes when $d_z(x)<R-r$, satisfies $\theta_z\cdot\nabla A_z=g$, and has uniformly bounded derivatives through order four on bounded sets. If $x\in B_{r_1}$, the ray $z+t\theta_z(x)$ for $t<0$ lies outside $B_r$, so $A_z(x)=Bg(x,\theta_z(x))$.

Choose $\chi\in C_c^\infty(B_{4R})$ with $\chi=1$ on $B_{3R}$. Set $\Phi=G_k^{(0)}(\cdot,z)$ and $v=R_0(k)(g\Phi)$. Direct differentiation, using $\nabla\Phi=(ik-d_z^{-1})\theta_z\Phi$, gives
\[
(-\Delta-k^2)\left(\frac{i}{2k}\Phi\chi A_z-v\right)
=E_k+\Phi A_z\,\theta_z\cdot\nabla\chi,
\]
where
\[
E_k=-\frac{i}{2k}\Phi\Delta(\chi A_z)
+\frac{i}{kd_z}\Phi\,\theta_z\cdot\nabla(\chi A_z),
\qquad \|E_k\|_{L^2}\leq Ck^{-1}.
\]
All these amplitudes vanish near $z$. The expression in parentheses is outgoing, so the classical resolvent and Lemma \ref{lem1} give an $O(k^{-2})$ bound for $R_0(k)E_k$ in $L^2(B_{r_1})$.

The remaining term equals
\[
\frac{1}{16\pi^2}\int
e^{ik(|x-y|+|y-z|)}
\frac{A_z(y)\theta_z(y)\cdot\nabla\chi(y)}
{|x-y||y-z|}\,dy.
\]
On its support, $|y|\geq3R$, whereas $x,z\in\overline B_R$. Hence the gradient in $y$ of the phase is uniformly nonzero. Two integrations by parts with
$\nabla_y\varphi\cdot\nabla_y/(ik|\nabla_y\varphi|^2)$, where $\varphi=|x-y|+|y-z|$, yield a uniform $O(k^{-2})$ bound. This proves the result.
\end{proof}

We next introduce the statistical data and the geometry used for stability. Let
\[
\Omega_e=B_{r_1}\setminus\overline B_r,\qquad
Y=L^2(\Omega_e\times\partial B_R,dx\,dS(z)
).
\]
By Corollary \ref{cor:mild-correlation}, the exterior correlation data are given by
\begin{equation}\label{eq:data-map}
\mathcal D_\rho(x,z;k)=\mathbb E[u(x,k)u(z,k)]
=\int_{B_r}G_k(x,y)G_k(z,y)\rho(y)\,dy,\qquad \rho=\sigma^2.
\end{equation}
 The identity \eqref{eq:data-map} defines $\mathcal D_\rho(x,z;k)$ for all $x,z\in\mathbb R^3\setminus\overline B_r$ and $k>0$.
For stability we restrict these data to $\Omega_e\times\partial B_R$.
They are ensemble second moments of the complex field, not intensity-only observations. The integral defines a linear map $\mathcal D_g$ for signed functions $g$ as well. In particular,
\[
\delta\mathcal D=\mathcal D_{\rho_1}-\mathcal D_{\rho_2}=\mathcal D_h.
\]

The nonlocality of the fractional Laplacian yields uniqueness at a single frequency. In fact, one receiver may be fixed, and the other receiver only
needs to range over an arbitrary nonempty open subset of the exterior.

\begin{theorem}\label{thm:stat-unique}
Let $0<\alpha<1$, $k>0$, and let $\rho_j=\sigma_j^2$, where
$\sigma_j\in C_c^\infty(B_r)$, $j=1,2$. Fix
$z_0\in\mathbb R^3\setminus\overline B_r$ and a nonempty open set
$U\subset\mathbb R^3\setminus\overline B_r$. If
\[
\mathcal D_{\rho_1}(x,z_0;k)=\mathcal D_{\rho_2}(x,z_0;k),
\qquad x\in U,
\]
then $\rho_1=\rho_2$. If both amplitudes are nonnegative, then
$\sigma_1=\sigma_2$ as well.
\end{theorem}

\begin{proof}
Set $h=\rho_1-\rho_2$ and
\[
F(y)=h(y)G_{\alpha,k}(z_0-y),\qquad
v=G_{\alpha,k}*F.
\]
Since $z_0$ lies outside the source region, $F\in C_c^\infty(B_r;\mathbb C)$. By \eqref{eq:data-map}, $v=0$ in $U$.
The outgoing field $v$ need not belong to $L^2(\mathbb R^3)$, so we first set $w=(-\Delta-k^2)v$. Cancellation of the spectral singularity gives
\begin{equation}\label{eq:single-frequency-multiplier}
\widehat w(\xi)=q_{\alpha,k}(\xi)\widehat F(\xi),\qquad
q_{\alpha,k}(\xi)=
\begin{cases}
\displaystyle\frac{|\xi|^2-k^2}{|\xi|^{2\alpha}-k^{2\alpha}},
&|\xi|\ne k,\\[8pt]
\displaystyle\frac{k^{2-2\alpha}}{\alpha},&|\xi|=k.
\end{cases}
\end{equation}
Indeed, near $|\xi|=k$ this follows from $p_{\alpha,k}(p_{\alpha,k}-i0)^{-1}=1$, while near the origin all factors are locally bounded functions. The multiplier $q_{\alpha,k}$ is strictly positive and satisfies
$|q_{\alpha,k}(\xi)|\leq C(1+|\xi|)^{2-2\alpha}$. Consequently, $w\in H^t(\mathbb R^3)$ for every $t\in\mathbb R$.
Moreover, \eqref{eq:single-frequency-multiplier} implies
\[
\big(( -\Delta)^\alpha-k^{2\alpha}\big)w=(-\Delta-k^2)F.
\]
Since $v=0$ and $F=0$ in $U$, we have $w=(-\Delta)^\alpha w=0$ there. The unique continuation property of the fractional Laplacian \cite[Theorem 1.2]{gsu}, applied to the real and imaginary parts, yields $w=0$ in $\mathbb R^3$. Since $q_{\alpha,k}>0$, it follows that $F=0$.

By \eqref{eq:mild-green-decomposition}--\eqref{eq:mild-J},
\[
\operatorname{Im}G_{\alpha,k}(z_0-y)
=\frac{k^{2-2\alpha}\sin(k|z_0-y|)}
{4\pi\alpha|z_0-y|}.
\]
Thus the zero set of $G_{\alpha,k}(z_0-\cdot)$ is contained in the countable union of spheres $|z_0-y|=n\pi/k$, $n\geq1$, which has Lebesgue measure zero. The identity $F=hG_{\alpha,k}(z_0-\cdot)=0$
therefore gives that $h=0$, completing the proof.
\end{proof}

\begin{remark}
The preceding theorem uses the correlation data $\mathcal D_\rho(\cdot,z_0;k)$ on an exterior open set. Multifrequency data on the surrounding shell and
receiver sphere will be used below to obtain quantitative stability.
\end{remark}

 We now derive the high-frequency estimate used for stability. Write $a_k=k^{2-2\alpha}/\alpha$. The Green function decomposition in Section \ref{sec:green} gives
\begin{equation}\label{gae}
G_k(x,y)=a_kG_k^{(0)}(x,y)+J_{\alpha,k}(|x-y|),
\qquad
|J_{\alpha,k}(d)|\leq Ck^{-4\alpha}d^{-3-2\alpha}.
\end{equation}
The bound follows directly from \eqref{eq:mild-J} and holds for every $d>0$.

The estimates below are uniform on $\mathcal C_F$, even if the support approaches $\partial B_r$. Indeed, the zero extension and Taylor's formula give
\[
|g(y)|\leq CF|x-y|^4,\qquad x\notin B_r,\quad y\in B_r,
\]
and consequently
\begin{equation}\label{eq:offdiag-weight}
\sup_{x\in\Omega_e}\int_{B_r}
\frac{|g(y)|}{|x-y|^{3+2\alpha}}\,dy\leq CF.
\end{equation}
Since $|z-y|\geq R-r$ for $z\in\partial B_R$, the mixed terms in the product of the two kernels in \eqref{eq:data-map} are $O(k^{2-6\alpha})$ and the product of the two $J$ terms is $O(k^{-8\alpha})$, uniformly in $x,z$. Applying Lemma \ref{he} to the classical term yields
\begin{equation}\label{crucial_1}
\mathcal D_g(x,z;k)
=\frac{i\,k^{3-4\alpha}}{2\alpha^2}G_k^{(0)}(x,z)Bg(x,\theta)+Q_{g,k}(x,z),
\end{equation}
where $\theta=(x-z)/|x-z|$ and
\[
\|Q_{g,k}\|_Y
\leq C\bigl(k^{2-4\alpha}+k^{2-6\alpha}+k^{-8\alpha}\bigr)
\leq Ck^{2-4\alpha},\qquad k\geq1.
\]
The three terms are, respectively, the classical geometric optics remainder, the mixed kernel terms, and the product of the two fractional corrections. After multiplication by $k^{4\alpha-3}$, they are bounded by $Ck^{-1}$, $Ck^{-1-2\alpha}$, and $Ck^{-3-4\alpha}$. These bounds hold for every $0<\alpha<1$. Here and below the same estimates apply to $h\in\mathcal C_{2F}$ with adjusted constants. As $R-r_1\leq|x-z|\leq R+r_1$, division by the known factor $G_k^{(0)}(x,z)$ gives
\begin{equation}\label{estimate}
\int_{\partial B_R}\int_{\Omega_e}
\left|Bg\left(x,\frac{x-z}{|x-z|}\right)\right|^2dx\,dS(z)
\lesssim\|k^{4\alpha-3}\mathcal D_g(k)\|_Y^2+k^{-2}.
\end{equation}

 We now turn to stability. For $g\in\mathcal C_F$, define its X-ray transform by $Xg(x,\theta)=\int_{\mathbb R}g(x+t\theta)\,dt$. Parametrize lines by $(p,\theta)$ with $p\in\theta^\perp$. For $|p|<r$, the total length of the line in $\Omega_e$ is
\[
\ell(p)=2\left(\sqrt{r_1^2-|p|^2}-\sqrt{r^2-|p|^2}\right).
\]
It is bounded above and below by positive constants depending only on $r,r_1$. Since lines with $|p|\geq r$ do not meet the support of $g$, Fubini's theorem gives
\begin{equation}\label{LP}
\int_{\Omega_e\times\mathbb S^2}|Xg(x,\theta)|^2\,dx\,d\theta
=\int_{\mathbb S^2}\int_{\substack{p\in\theta^\perp\\|p|<r}}
\ell(p)|Xg(p,\theta)|^2\,dp\,d\theta.
\end{equation}
The Fourier slice identity (see \cite{SU}) and Plancherel's theorem yield
\[
\int_{\mathbb S^2}\int_{\theta^\perp}|Xg(p,\theta)|^2\,dp\,d\theta
=c\int_{\mathbb R^3}|\xi|^{-1}|\widehat g(\xi)|^2\,d\xi
\geq c'\|g\|_{H^{-1/2}}^2.
\]
Together with $Xg(x,\theta)=Bg(x,\theta)+Bg(x,-\theta)$, this proves
\begin{equation}\label{estimate_2}
\|g\|_{H^{-1/2}}^2
\lesssim\|Xg\|_{L^2(\Omega_e\times\mathbb S^2)}^2
\lesssim\|Bg\|_{L^2(\Omega_e\times\mathbb S^2)}^2.
\end{equation}
For the change of variables $\theta=(x-z)/|x-z|$, the surface measure satisfies
\[
d\theta=\frac{R-x\cdot\nu(z)}{|x-z|^3}\,dS(z),\qquad \nu(z)=z/R.
\]
Its density has positive upper and lower bounds because $|x|\leq r_1<R$. Combining this observation with \eqref{estimate} and \eqref{estimate_2} gives the key estimate
\begin{equation}\label{crucial}
\|h\|_{H^{-1/2}(\mathbb R^3)}^2
\lesssim\|k^{4\alpha-3}\delta\mathcal D(k)\|_Y^2+k^{-2},
\qquad k\geq1.
\end{equation}

To continue the data beyond the measured interval, we use a sectorial analytic continuation argument; see \cite{WXZ_1} for its application to inverse random source problems. Since our measured interval is bounded away from zero, we record the following version with an interior slit. Fix $0<K_0<K_*$ with $K_*>1$, set $a=K_0/2$, and define
\begin{equation}\label{regionR}
\mathcal S_a=\{\zeta\in\mathbb C:|\arg(\zeta-a)|<\pi/4\}.
\end{equation}

\begin{lemma}\label{ac}
Let $K\ge K_*$ and suppose that $p$ is holomorphic in $\mathcal S_a$, with $|p|\le1$ there. If $|p(k)|\le\delta$ for $k\in[K_0,K]$, where $0<\delta<1$, then
\begin{equation}\label{p}
 |p(k)|\le\delta^{\,c(K/k)^2},\qquad k>K,
\end{equation}
where $c>0$ depends only on $K_0$ and $K_*$.
\end{lemma}

\begin{proof}
Write $t=(\zeta-a)/(K-a)$ and $b=(K_0-a)/(K-a)$. The sector becomes $\mathcal S_0$, and the measured interval becomes $[b,1]$, where
\[
 0<b\le b_*:=\frac{K_0-a}{K_*-a}<1.
\]
The conformal map $w=t^2$ sends $\mathcal S_0$ onto the right half-plane $\mathbb H=\{\Re w>0\}$ and the slit onto $[b^2,1]$. Let $\omega_b$ be the harmonic measure of this slit in $\mathbb H\setminus[b^2,1]$.

Put $B=b_*^2$ and consider the Green potential
\[
 P(w)=\int_B^1\log\left|\frac{w+s}{w-s}\right|\,ds.
\]
The function $P$ is nonnegative and harmonic in $\mathbb H\setminus[B,1]$, vanishes on the imaginary axis and at infinity, and extends continuously across $[B,1]$ because its logarithmic singularities are integrable. Consequently, $M_0:=\sup_{\mathbb H}P$ is finite and positive. Since $[B,1]\subset[b^2,1]$, the maximum principle gives
\[
 \omega_b(w)\ge M_0^{-1}P(w),
 \qquad w\in\mathbb H\setminus[b^2,1].
\]
For real $w>1$, the inequality $\log((1+q)/(1-q))\ge2q$, $0\le q<1$, yields
\[
 P(w)\ge\frac2w\int_B^1s\,ds=\frac{1-B^2}{w}.
\]
The two-constants theorem applied in the slit domain therefore gives
\[
 |p(k)|\le\delta^{\,c_0((K-a)/(k-a))^2},\qquad k>K.
\]
Finally,
\[
 \frac{K-a}{k-a}\ge\left(1-\frac a{K_*}\right)\frac Kk,
\]
which proves the assertion. All comparison constants depend only on the fixed quantities $K_0$ and $K_*$.
\end{proof}

Figure~\ref{fig:frequency-regions} shows the sector and the measured interval.

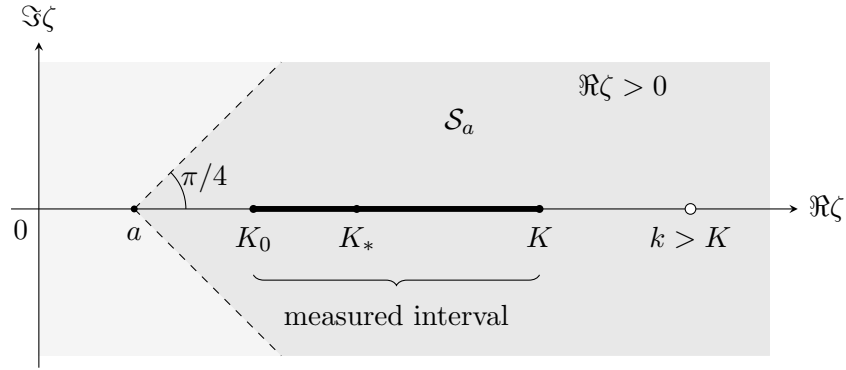
\begin{figure}[H]
\centering
\begin{tikzpicture}[x=1.05cm,y=1.05cm,font=\small,>=stealth]
  \fill[black!4] (0,-1.85) rectangle (9.2,1.85);
  \fill[black!9] (1.2,0) -- (3.05,1.85) -- (9.2,1.85)
    -- (9.2,-1.85) -- (3.05,-1.85) -- cycle;
  \draw[dashed,black] (1.2,0) -- (3.05,1.85);
  \draw[dashed,black] (1.2,0) -- (3.05,-1.85);
  \draw[->] (-0.35,0) -- (9.55,0) node[right] {$\Re\zeta$};
  \draw[->] (0,-2.0) -- (0,2.1) node[above] {$\Im\zeta$};
  \draw (1.85,0) arc[start angle=0,end angle=45,radius=0.65];
  \node at (2.08,0.39) {$\pi/4$};
  \node[black] at (5.3,1.1) {$\mathcal S_a$};
  \node at (7.35,1.53) {$\Re\zeta>0$};
  \node[below left] at (0,0) {$0$};
  \fill (1.2,0) circle (1.3pt);
  \node[below] at (1.2,-0.09) {$a$};
  \draw[line width=2.2pt,black] (2.7,0) -- (6.3,0);
  \foreach \x/\lab in {2.7/K_0,4.0/K_*,6.3/K}{
    \fill (\x,0) circle (1.5pt);
    \node[below] at (\x,-0.09) {$\lab$};
  }
  \draw[fill=white] (8.2,0) circle (2pt);
  \node[below] at (8.2,-0.09) {$k>K$};
  \draw[decorate,decoration={brace,mirror,amplitude=4pt}]
    (2.7,-0.83) -- (6.3,-0.83)
    node[midway,below=7pt] {measured interval};
\end{tikzpicture}
\caption{The translated sector with vertex $a=K_0/2$ and the interior
slit $[K_0,K]$. The map $w=((\zeta-a)/(K-a))^2$ sends the sector
to a right half-plane and the slit to $[b^2,1]$.}
\label{fig:frequency-regions}
\end{figure}

\begin{theorem}\label{main}
Let $\rho_j=\sigma_j^2\in\mathcal C_F$, $0<\alpha<1$, and $K\ge K_*$. Define
\begin{equation}\label{eq:discrepancy}
 \varepsilon^2=\sup_{k\in[K_0,K]}
 \|k^{4\alpha-3}(\mathcal D_{\rho_1}(k)-\mathcal D_{\rho_2}(k))\|_Y^2.
\end{equation}
There is a constant $C>0$, depending only on $\alpha,F,s,r,r_1,R,K_0,K_*$, such that, for $0<\varepsilon<e^{-1}$,
\begin{equation}\label{stability}
 \|\rho_1-\rho_2\|_{H^{-1/2}(\mathbb R^3)}^2
 \le C\left[\varepsilon^2+
 \min\left\{K^{-2},\frac1{K^{4/3}|\ln\varepsilon|^{2/3}}\right\}\right].
\end{equation}
The constant is independent of $K$ and $\varepsilon$.
\end{theorem}

\begin{proof}
Let $h=\rho_1-\rho_2$. All complex powers below use the branch in the right half-plane. The sector $\mathcal S_a$ is contained in $\{|\arg\zeta|<\pi/4\}$, and $|\zeta|\ge a$. The kernel formula \eqref{formula2} and the denominator estimate in the Appendix give
\[
 |G_\zeta(x,y)|
 \le C\left(
 \frac{(1+|\zeta|)^{2-2\alpha}e^{|\Im\zeta||x-y|}}{|x-y|}
 +\frac{(1+|\zeta|)^{-4\alpha}}{|x-y|^{3+2\alpha}}
 \right),\qquad \zeta\in\mathcal S_a.
\]
Using \eqref{eq:offdiag-weight}, the separation $|z-y|\ge R-r$, and the boundedness of the source and observation sets, we obtain
\[
 \|\mathcal D_h(\zeta)\|_Y
 \le C(1+|\zeta|)^{4-4\alpha}e^{T|\Im\zeta|}
\]
for a fixed $T>0$. The same bounds on compact subsets of the sector, followed by the scalar Cauchy formula, justify differentiation under the integrals and show that $\mathcal D_h$ is $Y$-valued holomorphic. Therefore
\[
 \mathcal F_h(\zeta)=\zeta^{4\alpha-3}\mathcal D_h(\zeta)
\]
is holomorphic and satisfies
\[
 \|\mathcal F_h(\zeta)\|_Y\le C(1+|\zeta|)e^{T|\Im\zeta|}.
\]
Since $\mathcal S_a$ is invariant under complex conjugation, the reflected function
$\mathcal F_h^\#(\zeta)=\overline{\mathcal F_h(\overline\zeta)}$ is holomorphic there. Set
\[
 I(\zeta)=\int_{\Omega_e\times\partial B_R}
 \mathcal F_h(x,z;\zeta)\mathcal F_h^\#(x,z;\zeta)\,dx\,dS(z).
\]
The bilinear product is holomorphic, with
\[
 I(k)=\|k^{4\alpha-3}\delta\mathcal D(k)\|_Y^2,\quad k\ge K_0,
 \qquad |I(\zeta)|\le C(1+|\zeta|)^2e^{2T|\Im\zeta|}.
\]
Choose $A\ge2T$ and then $M\ge1$ sufficiently large. Since
$|\Im\zeta|<\Re\zeta-a$ and $|1+\zeta|\ge(1+|\zeta|)/\sqrt2$ in $\mathcal S_a$, the function
\[
 p(\zeta)=\frac{e^{-A\zeta}I(\zeta)}{M(1+\zeta)^2}
\]
is bounded by one there. On $[K_0,K]$, it satisfies $|p|\le\varepsilon^2$. Lemma~\ref{ac} yields
\begin{equation}\label{eq:continued-data}
 I(k)\le M(1+k)^2
 \exp\left\{Ak-2c|\ln\varepsilon|\frac{K^2}{k^2}\right\},
 \qquad k>K.
\end{equation}

Put $E=|\ln\varepsilon|>1$. Choose a fixed $\eta\in(0,1]$ such that
$2c\eta^{-3}\ge A+1$ and set $q=\eta(K^2E)^{1/3}$.
If $q>K$, then \eqref{eq:continued-data} gives
\[
 I(q)\le M(1+q)^2e^{-q}\le Cq^{-2}.
\]
Applying \eqref{crucial} at $q$ therefore gives
\[
 \|h\|_{H^{-1/2}}^2\le Cq^{-2}
 \le CK^{-4/3}E^{-2/3}.
\]
If $q\le K$, then \eqref{crucial} at the measured frequency $K$ gives
\[
 \|h\|_{H^{-1/2}}^2\le C(\varepsilon^2+K^{-2})
 \le C(\varepsilon^2+K^{-4/3}E^{-2/3}),
\]
where the last step uses $E/K\le\eta^{-3}$. Finally, the measured-frequency estimate
$\|h\|_{H^{-1/2}}^2\le C(\varepsilon^2+K^{-2})$ holds in both cases. Taking the smaller of these bounds proves \eqref{stability}.
\end{proof}

\begin{remark}
The estimate combines a Lipschitz data term with a logarithmic remainder. For a fixed normalized discrepancy, the remainder decreases as the upper frequency $K$ increases. The minimum in \eqref{stability} also retains the direct estimate available at the largest measured frequency. No uniformity is asserted as $\alpha$ approaches either endpoint of $(0,1)$.
\end{remark}

\begin{corollary}\label{cor:L2-stability}
Under the hypotheses of Theorem~\ref{main},
\[
 \|\rho_1-\rho_2\|_{L^2(\mathbb R^3)}^2
 \le C\left(\varepsilon^2+
 \min\left\{K^{-2},K^{-4/3}|\ln\varepsilon|^{-2/3}\right\}\right)^{\frac{2s}{2s+1}}.
\]
\end{corollary}
\begin{proof}
The common support and the $W^{s,\infty}$ bound imply $\|h\|_{H^s}\le C$. H\"older's inequality in Fourier space gives
\[
 \|h\|_{L^2}^2\le
 \|h\|_{H^{-1/2}}^{\frac{4s}{2s+1}}
 \|h\|_{H^s}^{\frac2{2s+1}}.
\]
The result follows from Theorem~\ref{main}.
\end{proof}

 \begin{remark}
For an inverse potential problem, quantitative unique continuation results such as \cite{grsu,ruland} suggest another route from exterior observations to stability. Their use for outgoing fractional Helmholtz fields would require estimates in the appropriate solution and observation spaces. 
\end{remark}

To conclude this section, we briefly discuss far-field data. For fixed $k>0$, the Green function expansion gives, almost surely,
\[
u(x,k)=\frac{k^{2-2\alpha}}{\alpha}
\frac{e^{ik|x|}}{4\pi|x|}u^\infty(\hat x,k)+ O(|x|^{-2}),
\qquad \hat x=x/|x|,
\]
uniformly in $\hat x$, where
\[
u^\infty(\hat x,k)=\int_{\mathbb R^3}e^{-ik\hat x\cdot y}\sigma(y)\,dW_y.
\]
 The $O(|x|^{-2})$ term includes the next classical spherical-wave term; the fractional correction alone is $O(|x|^{-3-2\alpha})$. With the same source realization used across frequencies, the white-noise isometry gives
\[
\mathbb E[u^\infty(\hat x,k+\tau)\overline{u^\infty(\hat x,k)}]
=\widehat{\sigma^2}(\tau\hat x),\qquad  k,k+\tau>0.
\]
This identity connects the far-field correlation directly to the Fourier transform of the variance, as in the classical inverse source analysis \cite{blz}.

\section{Conclusion}


 We have investigated the direct and inverse random source problems for the fractional Helmholtz equation. For every \(0<\alpha<1\), we establish the existence and uniqueness of the outgoing distributional solution, together with its exterior stochastic representation and high-frequency resolvent estimates. When \(\alpha>3/4\), this solution further admits a continuous mild realization in the whole space. For the inverse problem, we show that the variance  is uniquely determined at a single frequency from exterior correlation data with one receiver fixed and the other varying over an arbitrary nonempty exterior open set. For multifrequency measurements on an exterior shell and a surrounding sphere, the high-frequency expansion of the fractional Green function, combined with a geometric optics construction, connects the correlation data to the X-ray transform of the variance. This connection, together with the stability properties of the X-ray transform and analytic continuation in the frequency variable, leads to an increasing stability estimate as the upper endpoint of the measured frequency interval grows. A natural direction for further study is inverse random potential scattering for the fractional Schr\"odinger equation.

\appendix
\section{Proofs of Lemma \ref{lem1} and Proposition \ref{free-frac}}

\subsection{Proof of Lemma \ref{lem1}}

Define the sine propagator
\[
U(t)=\frac{\sin(t\sqrt{-\Delta})}{\sqrt{-\Delta}}.
\]
For $\Im\lambda>0$, the spectral theorem gives
\[
R_0(\lambda)=\int_0^\infty e^{i\lambda t}U(t)\,dt.
\]
 The Fourier multiplier also shows that this is holomorphic from $H^s$ to $H^{s+2}$ in the upper half-plane; intermediate orders follow by interpolation.

Fix $L>\operatorname{diam}(\operatorname{supp}\chi)$. In dimension three, the strong Huygens principle gives
\[
\chi R_0(\lambda)\chi=\int_0^L e^{i\lambda t}\chi U(t)\chi\,dt.
\]
 For $0\leq t\leq L$, the spectral theorem yields
\[
\|U(t)\|_{H^s\to H^{s+1}}\leq C_L,\qquad
\|\partial_tU(t)\|_{H^s\to H^s}\leq1.
\]
Thus the finite integral is entire and has an $H^s\to H^{s+1}$ bound $Ce^{L(\Im\lambda)_-}$. Since its endpoint terms vanish, integration by parts gives, for $\lambda\neq0$,
\[
\chi R_0(\lambda)\chi
=-\frac{1}{i\lambda}\int_0^L e^{i\lambda t}\chi\partial_tU(t)\chi\,dt.
\]
 Together with the direct integral for $|\lambda|\leq1$, this proves \eqref{res1} for $m=0,1$.

 For $m=2$, use
\[
\Delta(\chi R_0(\lambda)\chi)
=-\lambda^2\chi R_0(\lambda)\chi-\chi^2
+[\Delta,\chi]R_0(\lambda)\chi.
\]
Choose a cutoff $\chi_1=1$ near $\operatorname{supp}\chi$ whose support still has diameter less than $L$. The $m=1$ estimate for $\chi_1R_0(\lambda)\chi_1$ controls the commutator in $H^s$. The elliptic estimate for $1-\Delta$ therefore gives the $H^s\to H^{s+2}$ bound $C(1+|\lambda|)e^{L(\Im\lambda)_-}$ and the corresponding holomorphy. Interpolation proves the remaining cases.

\subsection{Proof of Proposition \ref{free-frac}}

Let
\[
D_\alpha(\gamma,\lambda)
=\gamma^{2\alpha}-2\gamma^\alpha\lambda^{2\alpha}\cos(\pi\alpha)+\lambda^{4\alpha}.
\]
 For $\lambda\in S_{\theta_0}$, the arguments of $\lambda^{2\alpha}$ lie strictly between $-\pi\alpha$ and $\pi\alpha$. The factorization
\[
D_\alpha(\gamma,\lambda)
=(\lambda^{2\alpha}-\gamma^\alpha e^{i\pi\alpha})
(\lambda^{2\alpha}-\gamma^\alpha e^{-i\pi\alpha})
\]
therefore has no zero for $\gamma\geq0$. Scaling by $|\lambda|^{4\alpha}$ and compactness of the angular interval give
\begin{equation}\label{eq:denominator}
|D_\alpha(\gamma,\lambda)|
\geq c_{\alpha,\theta_0}(\gamma^{2\alpha}+|\lambda|^{4\alpha}).
\end{equation}
Indeed, after setting $t=\gamma/|\lambda|^2$, the ratio to $1+t^{2\alpha}$ is continuous and positive on bounded $t$ intervals and tends uniformly to one as $t\to\infty$.

Denote the integral term in \eqref{formula2} by $J(\lambda)$. Its Fourier
multiplier is
\[
b_J(q,\lambda)=\frac{\sin(\pi\alpha)}{\pi}
\int_0^\infty
\frac{\gamma^\alpha}
{(\gamma+q^2)D_\alpha(\gamma,\lambda)}\,d\gamma,
\qquad q=|\xi|.
\]
For each $q\geq0$, the integral is absolutely convergent and holomorphic in $\Re\lambda>0$. In particular, at $q=0$ the integrand is bounded locally by $C\gamma^{\alpha-1}$ near zero and by $C\gamma^{-\alpha-1}$ near infinity.

Set $a=|\lambda|$ and $Q=q/a$. By \eqref{eq:denominator},
\[
|b_J(q,\lambda)|
\leq Ca^{-2\alpha}\int_0^\infty
\frac{t^\alpha}{(t+Q^2)(1+t^{2\alpha})}\,dt.
\]
If $0\leq Q\leq1$, this integral is bounded by $\int_0^\infty t^{\alpha-1}(1+t^{2\alpha})^{-1}\,dt<\infty$.
If $Q\geq1$, splitting it at $Q^2$ gives
\[
\begin{split}
\int_0^\infty\frac{t^\alpha}{(t+Q^2)(1+t^{2\alpha})}\,dt
&\leq Q^{-2}\int_0^{Q^2}\frac{t^\alpha}{1+t^{2\alpha}}\,dt
  +\int_{Q^2}^\infty t^{-\alpha-1}\,dt\\
&\leq C Q^{-2\alpha}.
\end{split}
\]
Consequently,
\begin{equation}\label{eq:J-multiplier}
|b_J(q,\lambda)|\leq C(|\lambda|^2+q^2)^{-\alpha},
\qquad \lambda\in S_{\theta_0}.
\end{equation}
Plancherel's theorem now gives, for $|\lambda|\geq1$ and
$0\leq l\leq2\alpha$,
\[
\|J(\lambda)\|_{H^s\to H^{s+l}}
\leq C\sup_{q\geq0}(1+q^2)^{l/2}(|\lambda|^2+q^2)^{-\alpha}
\leq C|\lambda|^{l-2\alpha}.
\]
Multiplication by $\rho$ is bounded on every Sobolev space, so the same bound holds for $\rho J(\lambda)\rho$.
For the classical term, Lemma~\ref{lem1} gives
\[
\left\|\frac{\lambda^{2-2\alpha}}{\alpha}
\rho R_0(\lambda)\rho\right\|_{H^s\to H^{s+l}}
\leq C|\lambda|^{1+l-2\alpha}e^{T(\Im\lambda)_-}.
\]
Combining the last two inequalities proves \eqref{bound_2}.

It remains to justify holomorphy at the endpoint $l=2\alpha$. On every compact subset of the right half-plane, \eqref{eq:J-multiplier} bounds $(1+q^2)^\alpha b_J(q,\lambda)$ uniformly in $q$ and $\lambda$.
Apply the scalar Cauchy formula on a small circle contained in that half-plane, for each fixed $q$. The resulting Cauchy estimates and Taylor remainders are uniform after multiplication by $(1+q^2)^\alpha$.
Thus $J(\lambda)$ is holomorphic in the operator norm of $\mathcal L(H^s,H^{s+2\alpha})$. The assertion for smaller $l$ follows by continuous embedding.
The classical truncated term is entire by Lemma~\ref{lem1}, while $\lambda^{2-2\alpha}$ is holomorphic in the right half-plane.
Formula~\eqref{formula2} therefore supplies the asserted continuation.
It agrees with the spectral resolvent in the first quadrant and with the outgoing boundary value on the positive real axis.


\begin{thebibliography}{99}

\bibitem{arridge}
S. Arridge, Optical tomography in medical imaging, Inverse Problems, 15 (1999), R41--R93.

\bibitem{BCL2016}
G. Bao, C. Chen, and P. Li, Inverse random source scattering problems in several dimensions, SIAM/ASA J. Uncertain. Quantif., 4 (2016), 1263--1287. 

\bibitem{blz}
G. Bao, P. Li, and Y. Zhao, Stability for the inverse source problems in elastic
and electromagnetic waves, J. Math. Pures Appl., 134 (2020), 122--178.






\bibitem{bl}
N. S. Belevtsov and S. Y. Lukashchuk, A fast algorithm for fractional Helmholtz equation with application to
electromagnetic waves propagation, Appl. Math. Comput., 416 (2022), 126728.



\bibitem{BV}
C. Bucur and E. Valdinoci, Non-local diffusion and applications, Lecture Notes of the Unione Matematica Italiana, 20, Springer, 2016.

\bibitem{CL}
X. Cao and H. Liu, Determining a fractional Helmholtz equation with unknown source and scattering potential, Commun. Math. Sci., 17(2019), 1861--1876.

\bibitem{DZ}
S. Dyatlov and M. Zworski, Mathematical Theory of Scattering Resonances, volume 200, American
Mathematical Soc., 2019.

\bibitem{sisc}
C. Glusa, H. Antil, M. D' Elia, B. van Bloemen Waanders, and C. J. Weiss, A fast solver for the fractional
Helmholtz equation, SIAM J. Sci. Comput., 43 (2021), A1362--A1388.

\bibitem{grsu}
T. Ghosh, A. R\"uland, M. Salo, and G. Uhlmann, Uniqueness and reconstruction for the fractional Calder$\rm \acute o$n
problem with a single measurement, J. Funct. Anal., 279 (2020), 108505.

\bibitem{gsu}
T. Ghosh, M. Salo, and G. Uhlmann, The Calder$\rm \acute o$n problem for the fractional Schr\"odinger equation, Anal. PDE,
13 (2020), 455--475.

\bibitem{HSZ}
X. Huang, Y. Sire, and C. Zhang, Interior estimates for the eigenfunctions of the fractional Laplacian on a bounded domain, Adv. Math., 392 (2021), 108032.

\bibitem{Isakov_1}
V. Isakov, Inverse Source Problems, AMS, Providence, RI, 1990.



\bibitem{S1}
J. Kraisler and J. C. Schotland, Collective spontaneous emission and kinetic equations for one-photon light in
random media, J. Math. Phys., 63 (2022), 031901.

\bibitem{S2}
J. Kraisler and J. C. Schotland, Kinetic equations for two-photon light in random media, J. Math. Phys., 64
(2023), 111903.

\bibitem{Laskin1}
Laskin, N. Fractional quantum mechanics, Phys. Rev. E 62,
3135--3145 (2000).

\bibitem{Laskin2}
Laskin, N. Fractional quantum mechanics and L$\rm \acute e$vy path integrals.
Phys. Lett. A 268, 298--305 (2000).

\bibitem{LPS}
M. Lassas, L. P{\"a}iv{\"a}rinta, and E. Saksman,
Inverse scattering problem for a two dimensional random potential,
Comm. Math. Phys., 279 (2008), 669--703.

\bibitem{LLW}
J. Li, P. Li, and X. Wang, Inverse source problems for the stochastic wave equations: far-field
patterns, SIAM J. Appl. Math., 82 (2022), 1113--1134.

\bibitem{LLM}
J. Li, H. Liu, and S. Ma,
Determining a random Schr\"odinger operator: both potential and source are random,
Comm. Math. Phys., 381 (2021), 527--556.

\bibitem{li2021inverse}
P. Li and X. Wang, Inverse random source scattering for the Helmholtz equation with attenuation.
SIAM J. Appl. Math., 81 (2021), 485--506.




\bibitem{LZZ}
P. Li, J. Zhai, and Y. Zhao, Stability for the acoustic inverse source problem in inhomogeneous media, SIAM J. Appl. Math., 80 (2020), 2547--2559.


\bibitem{LL}
P. Li and Z. Li, An inverse random source problem for the fractional Helmholtz equation, arXiv: 2602.19559v1.

\bibitem{msq}
S. Das, T. Ghosh, S. Ma,
Inverse scattering for the fractional Schr\"odinger equation,
arXiv:2509.12685, 2025.

\bibitem{MS}
C. Martinez and M. Sanz, The Theory of Fractional Powers of Operators, North Holland Math Stud. 187, Elsevier,
Amsterdam, 2001.


\bibitem{Ros}
X. Ros-Oton, Nonlocal elliptic equations in bounded domains: a survey, Publ. Mat. 60 (2016), 3--26.


\bibitem{ruland}
A. R\"uland, On single measurement stability for the fractional Calder$\rm \acute o$n problem, SIAM J. Math. Anal., 53 (2021),
5094--5113.

\bibitem{Stefanov}
P. Stefanov, Scattering and Inverse Scattering, Lecture Notes.

\bibitem{SU}
P. Stefanov and G. Uhlmann, Microlocal Analysis and Integral Geometry, American
Mathematical Soc., 2026.


\bibitem{TT}
V. E. Tarasov and J. J. Trujillo, Fractional power-law spatial dispersion in electrodynamics, Ann. Physics, 334
(2013), 1--23.



\bibitem{WXZ_1}
T. Wang, X. Xu, and Y. Zhao, Stability for a multi-frequency inverse random source problem.
Inverse Problems, 40 (2024), 125029.


\bibitem{WXZ}
T. Wang, X. Xu, and Y. Zhao, Stability for an inverse random source problem of the biharmonic Schr\"odinger equation, arXiv:2412.16836, 2024.




\bibitem{ZCV}
D. Zilberberg, F. Cakoni, and M. S. Vogelius, Limiting absorption principle and radiation condition for the fractional Helmholtz equation, arXiv:2602.18387, 2026.

\end{thebibliography}
\end{document}